\documentclass{amsart}
\usepackage{amssymb,amsmath,amscd, amsfonts,latexsym}
\usepackage{graphicx}
\usepackage{url}
\usepackage[top = 3.75cm, bottom = 3.75cm, left = 3.7cm, right = 3.7cm]{geometry}
\usepackage{xcolor}
\usepackage{filecontents}
\usepackage{enumitem}
\usepackage{array}   
\newcolumntype{L}{>{$}l<{$}} 
\usepackage{float}
\usepackage[labelsep=period]{caption}

\usepackage{listings}
\usepackage{hyperref}
\hypersetup{colorlinks,
	linkcolor=green!50!black,
	citecolor=blue!70!black,
	urlcolor=red!50!black
}

\newtheorem{thm}{Theorem}
\newtheorem{theorem}[thm]{Theorem}

\newtheorem{lemma}[thm]{Lemma}

\newtheorem{prop}[thm]{Proposition}

\newtheorem{conjecture}[thm]{Conjecture}

	\newcounter{countknownthm}
	
\newtheorem{knownthm}[countknownthm]{Theorem}
	\newcounter{countknownlem}
	
\newtheorem{knownlem}[countknownlem]{Lemma}

\theoremstyle{definition}
\newtheorem{definition}[thm]{Definition}

\def \R {{\mathbb R}}
\def \a {{\mathfrak a}}

\def \N {{\mathbb N}}
\def \C {{\mathbb C}}
\def \Z {{\mathbb Z}}
\def \Q {{\mathbb Q}}

\DeclareMathOperator{\disc}{disc}

\def\eps{\varepsilon}
\def\Z{\mathbb{Z}}

\def\Q{\mathbb{Q}}
\def\R{\mathbb{R}}
\def\C{\mathbb{C}}
\def\p{\mathfrak{p}}
\def\OO{\mathcal{O}}

\def\bb{,\ldots,}

\title[Thue equations over $\C(T)$: The Complete Solution of a Simple quartic family]{Thue equations over $\C(T)$: The Complete Solution of a Simple quartic family}



\author[B. Faye]{Bernadette Faye}
\address{B. Faye,
	UFR SATIC, Universit\'e Alioune Diop de Bambey,
	Diourbel, Bambey 30, S\'en\'egal}
\email{bernadette.faye@uadb.edu.sn}

\author[I. Vukusic]{Ingrid Vukusic}
\address{I. Vukusic,
University of Salzburg,
Hellbrunnerstrasse 34/I,
A-5020 Salzburg, Austria.  Current address:  Department of Mathematics,
University of York,
York, North Yorkshire YO10 5GH,
United Kingdom.}
\email{ingrid.vukusic@york.ac.uk}

\author[E. Waxman]{Ezra Waxman}
\address{E. Waxman,
	University of Haifa, Department of Mathematics, 199 Aba Khoushy Ave., Mt. Carmel, Haifa, 3498838}
	\address{Unit of Mathematics, Afeka — The Academic College of Engineering in Tel Aviv, Mivtsa Kadesh St 38, Tel Aviv-Yafo 6998812, Israel}
\email{ezrawaxman@gmail.com}

\author[V. Ziegler]{Volker Ziegler}
\address{V. Ziegler,
University of Salzburg,
Hellbrunnerstrasse 34/I,
A-5020 Salzburg, Austria}
\email{volker.ziegler\char'100plus.ac.at}

\begin{document}


\begin{abstract}
In this paper we completely solve a simple quartic family of Thue equations over $\mathbb{C}(T)$.  Specifically, we apply the ABC-Theorem to find all solutions $(x,y) \in \mathbb{C}[T] \times \mathbb{C}[T]$ to the set of Thue equations $F_{\lambda}(X,Y) = \xi$, where $\xi \in \mathbb{C}^{\times}$  and
\begin{equation*}
F_{\lambda}(X,Y):=X^4 -\lambda X^3Y -6 X^2Y^2 + \lambda XY^3 +Y^4,  \quad \quad \lambda \in \mathbb{C}[T]/\{\mathbb{C}\}
\end{equation*}
denotes a family of quartic simple forms. 
\end{abstract}

\maketitle

\section{Introduction}
\textit{Diophantine equations}, named after Diophantus of Alexandra, have been an enduring topic of mathematical interest from antiquity up until the modern era.  Pythagoras, for example, studied integer solutions to the equation $X^{2}+Y^{2}=Z^{2}$, while  Brahmagupta, Euler, and Fermat studied such solutions to the equation $61X^2 + 1 = Y^2$.  By the twentieth century, a much richer general theory of Diophantine equations began to emerge.  Axel Thue \cite{Thue1909}, for instance, considered equations of the form $F(X,Y) = m$, where  $m$ is a non-zero integer, and $F(X,Y) \in \Z[X,Y]$ is an irreducible homogeneous \textit{binary form} of degree $n \geq 3$.  In 1909, he managed to prove that such equations (now known as \textit{Thue equations}) have only finitely many integer solutions $(x,y) \in \Z^{2}$. Thue's result, however, was not \textit{effective}, i.e. did not provide a bound for the size of such solutions.  Baker \cite{Baker1966-68} resolved this in the 1960's, by developing powerful methods to compute lower bounds for linear forms in logarithms.  Such tools could then be applied to solve Thue equations effectively.  In other words, Baker's method managed to reduce, to a finite amount of computation, the 
problem of determining all integer solutions $(x,y) \in \Z^{2}$ to a given Thue equation.

\subsection{Families of Thue Equations}
One direction of investigation then turned towards studying parametrized \textit{families} of Thue equations.  E. Thomas \cite{Thomas1990},  for instance, considered the family of cubic forms

\begin{equation}\label{eq:cubic}
F^{(3)}_{t}(X,Y) := X^3-(t-1)X^2Y-(t+2)XY^{2}-Y^{3}
\end{equation}
for $t \in \Z_{\geq 0}$.  He conjectured that for $t \geq 4$, the Thue equation
\[F^{(3)}_{t}(X,Y) = \pm 1 \]
has only the ``trivial" solutions $(x,y) \in \{(0,\mp 1), (\pm 1,0), (\mp 1,\pm 1)\}$.  Such a conjecture was eventually proved correct by Mignotte \cite{Mignotte1993}.   More general questions related to such Thue equations were addressed in \cite{Gaal2019, LettlPethoVoutier1999}. Lettl and Peth\H{o} \cite{LettlPetho1995} then investigated the family of quartic forms
\begin{equation}\label{eq:quartic}
F_{t}^{4}(X,Y):=X^4 -tX^3Y -6X^2Y^2 + tXY^3 +Y^4
\end{equation}
and determined the complete solution set for Thue equations of the form $F_{t}^{4}(X,Y) = m$, where $t \in \Z$ and $m \in \{\pm 1, \pm 4\}$.  The families in \eqref{eq:cubic}
and \eqref{eq:quartic} are known as \textit{simple forms}, and are discussed below in Section \ref{sec:Simple forms} in further detail.  For a general survey discussion about families of Thue equations see \cite{Heuberger2006}.

\subsection{Thue Equations Over Function Fields}
One may also consider Thue equations in the function field setting.  More precisely, we consider equations of the form $F(X,Y) = m$, for some non-zero $m \in \C[T]$, where
\begin{equation*}
F(X,Y) = a_{0}X^{n} + a_{1}X^{n-1}Y+\cdots + a_{n-1}X Y^{n-1}+a_{n}Y^{n}, \hspace{5mm} a_{i} \in \C[T],
\end{equation*}
is irreducible of degree $n \geq 3$, and where we now seek solutions $(x,y) \in \C[T] \times \C[T]$.  By applying a function field analogue of Thue's method, Gill \cite{Gill1930} demonstrated that the solutions to any such equation have bounded degree. Using methods developed by Osgood \cite{Osgood1973}, Schmidt managed to obtain explicit bounds on the degree of such solutions.  In contrast to classical Thue equations, however, such a bound does not directly imply that only finitely many
such solutions exist.  Mason \cite{Mason1981, Mason1984} eventually succeeded in demonstrating that the solution set of a Thue equation over $\C(T)$ may be effectively determined.  For a history on the development of Thue equations over function fields see \cite{Mason1983}.

Families of Thue equations over $\C(T)$ were first discussed in \cite{FuchsZiegler2006}, and the $\C(T)$ analogue of \eqref{eq:cubic} was resolved in \cite{FuchsZiegler}.  The purpose of this work is to investigate the $\C(T)$ analogue of \eqref{eq:quartic}.  We obtain the following result:

\begin{theorem}\label{thm:main}
Fix a non-constant $\lambda \in \C[T]$, and consider the (homogeneous) polynomial
\begin{equation}\label{eq:FF_Forms}
F_{\lambda}(X,Y) :=X^4 -\lambda X^3Y -6 X^2Y^2 + \lambda XY^3 +Y^4.
\end{equation}
Then for any $\xi \in \C^{\times}$ the solution set of the Thue equation 
$$F_{\lambda}(X,Y) =\xi$$
is equal to
\begin{align*}
S_{\lambda,\xi}&:=\{(x,y)\in \C[T] \times \C[T]: F_{\lambda}(x,y)= \xi\}\\
&\phantom{:}=\{(\eta,0),(0,\eta): \eta^{4} = \xi \} \cup \{(\eta,\eta), (\eta,-\eta): -4\eta^{4} = \xi \}.
\end{align*}
\end{theorem}

\subsection{Simple Forms}\label{sec:Simple forms}
To motivate the study of simple forms, consider the \textit{M\"{o}bius} map $\phi: z \mapsto \frac{az+b}{cz+d}$, with $a,b,c,d \in \Z$.
Let $G_{\phi} =\langle \phi \rangle$ denote the cyclic group generated by $\phi$.  If $\phi$ has finite order, it may be shown that $|G_{\phi}| \in \{1,2,3,4,6\}$.  Let $\phi$ be a \textit{M\"{o}bius} map of finite order, and suppose there exists an irreducible form $F(X,Y) \in \Z[X,Y]$ of degree $n \in \{3,4,6\}$ such that $G_{\phi}$ acts transitively on the roots of $F(X,1)$.
Lettl, Peth\H{o}, and Voutier \cite{LettlPethoVoutier1999} refer to such forms as \textit{simple forms}.

As an example, consider the map $\phi: z \mapsto \frac{-1}{z+1}$, which generates a cyclic group $G_{\phi}$ of order 3.  We ask for the set of irreducible cubic polynomials $f(X)$ upon whose roots $G_{\phi}$ acts transitively. Such polynomials must be of the form
\begin{align*}
f^{(3)}_{t}(X) &= (X-\alpha)(X-\phi(\alpha))(X-\phi^{2}(\alpha))\\
&= X^{3}+\left(\frac{1}{\alpha}+\frac{1}{1+\alpha}-\alpha+1\right)X^{2}+\left(\frac{1}{\alpha}+\frac{1}{1+\alpha}-\alpha-2\right)X-1\\
&=X^3-(t-1)X^2-(t+2)X-1
\end{align*}
where $\alpha$ denotes a root of $f_{t}^{(3)}(X)$, and
where $t:= \alpha-\frac{1}{\alpha}-\frac{1}{1+\alpha}$.  We then obtain the family of simple cubic forms in \eqref{eq:cubic} upon restricting $t \in \Z_{\geq 0}$.

Two forms $F(X,Y),G(X,Y) \in \Q[X,Y]$ are said to be \textit{equivalent} if there exists a $t \in \Q^{\times}$ and a matrix $\begin{pmatrix}
p & q\\
r & s
\end{pmatrix}
 \in GL_{2}(\Q)$
such that $G(X,Y) = t \cdot F(pX+qY,rX+sY)$.  It may be demonstrated that any simple form is equivalent to a form in one of the following two parameter families:
\begin{align*}\label{eq:two_parameter_simple_forms}
\begin{split}
F_{s,t}^{(3)}(X,Y) &= sX^{3}-(t-s) X^{2}Y -(t+2s)XY^{2}-sY^{3},\\
F_{s,t}^{(4)}(X,Y) &= sX^{4}-t X^{3}Y - 6sX^{2}Y^{2}+tXY^{3}+sY^{4},\\
F_{s,t}^{(6)}(X,Y) &= s X^{6}-2tX^{5}Y - (5t+15s)X^{4}Y^{2}-20sX^{3}Y^{3}+5tX^{2}Y^{4}+(2t+6s)XY^{5}+sY^{6}.
\end{split}
\end{align*}
Above we only consider irreducible such forms, and moreover restrict $s \in \N$, $t \in \Z$ such that $(s,t)=1$.  These two-parameter families of forms have been studied in \cite{Wakabayashi2007} by applying the hypergeometric method.

When $s=1$, the corresponding polynomial $f^{(i)}_{t}(X):=F^{(i)}_{1,t}(X,1)$ is monic with constant term $\pm 1$, which enables an easier application of Baker's method to the study of such forms.  Note that the family of cubic forms $F^{(3)}_{1,t}(X,Y)$, $t \in \Z_{\geq 0}$, corresponds to those in \eqref{eq:cubic}, while the family of quartic forms $F^{(3)}_{1,t}(X,Y)$, $t \in \Z$, corresponds to those in \eqref{eq:quartic}.

\subsection{Solving Thue Equations: Siegel's Identity and $S$-Unit Equations}\label{sec:siegel_intro}
The method for solving Thue equations in both the number field and function field settings begins similarly.  We specialize to the case where $A$ denotes either the ring $\Z$ or the ring $\C[T]$.  Let $(x,y)\in A^{2}$ denote a solution to the Thue equation
\begin{equation}\label{eq:Thue_equation}
F(X,Y) = m,
\end{equation}
where $F(X,Y) \in A[X,Y]$ is a homogeneous form of degree $n\geq 3$, and $m \in A$ is non-zero.  For simplicity, we moreover assume that $f(X):=F(X,1)$ is monic, so that we may factor 
\begin{equation}\label{eq:factorization_in_beta}
F(x,y) = (x-\alpha_{1} y)\dots (x-\alpha_{n} y) = m,
\end{equation}
where $\alpha_{1},\dots,\alpha_{n}$ denote the roots of $f(X)$.

Let $k$ denote the fraction field of $A$ (i.e. either $\Q$ or $\C(T)$), and let $K$ denote the splitting field of $f(X)$ over $k$.  We moreover use $\mathcal{O}_{K}$ to denote the ring of integers of $K$, that is $\mathcal{O}_{K}$ denotes the integral closure of $A$ in $K$.  From \eqref{eq:factorization_in_beta} it follows that $\beta_{i}:= x - \alpha_{i}y$ are \textit{$S$-units} in $\mathcal{O}_{K}$, where $S$ denotes the set of prime ideals in $\mathcal{O}_{K}$ that lie above either a prime dividing $m$ or the prime at infinity.  By \textit{Siegel's identity} we moreover find that
\begin{equation*}
-\frac{(\alpha_{2}-\alpha_{3})}{(\alpha_{1}-\alpha_{2})}\frac{\beta_{1}}{\beta_{3}} - \frac{(\alpha_{3}-\alpha_{1})}{(\alpha_{1}-\alpha_{2})}\frac{\beta_{2}}{\beta_{3}} =1.
\end{equation*}
Upon setting $u_{1}:=-\frac{(\alpha_{2}-\alpha_{3})}{(\alpha_{1}-\alpha_{2})}\frac{\beta_{1}}{\beta_{3}}$ and $u_{2}:=- \frac{(\alpha_{3}-\alpha_{1})}{(\alpha_{1}-\alpha_{2})}\frac{\beta_{2}}{\beta_{3}}$, we thus obtain a solution to the \textit{$S$-unit equation}
\begin{equation}\label{eq:S_unit}
u_{1}+u_{2} = 1,
\end{equation}
where $u_{1},u_{2} \in K$ are again $S$-units, where $S$ now moreover includes the finite set of primes in $K$ dividing $(\alpha_{2}-\alpha_{3}), (\alpha_{1}-\alpha_{2})$, or $(\alpha_{3}-\alpha_{1})$.

In the classical setting, one may 
use Baker's method of lower bounds for linear forms in logarithms to obtain an effective upper bound on the height of the possible solutions to such $S$-unit equations. Since each solution $(x,y) \in \Z^{2}$ of the Thue equation $F(X,Y) = m$ corresponds to a pair of $S$-units $(u_{1},u_{2}) \in K^{2}$ satisfying \eqref{eq:S_unit}, one may effectively determine the entire set of solutions to \eqref{eq:Thue_equation}.

\subsection{A $\C(T)$ Strategy for Solving Thue Equations: The ABC Conjecture}

One may alternatively obtain an upper bound on the height of the possible solutions to \eqref{eq:S_unit} by applying an appropriate form of the \textit{ABC conjecture}.  First formulated by Joseph Oesterlé and David Masser in 1985, the ABC conjecture is considered perhaps the most important unsolved problem in Diophantine analysis.  The classical version may be stated as follows: let $a,b,c \in \Z$, such that $a+b=c$, and suppose moreover that $a,b,$ and $c$ are pairwise co-prime.  Then for any $\epsilon > 0$, there exists a constant $M_{\epsilon}$ such that
\begin{equation*}
\max(|a|,|b|,|c|) \leq M_{\epsilon}\prod_{p|abc}p^{1+\epsilon}.
\end{equation*}

Recall that the \textit{height} of any $r \in \Q^{\times}$ is defined to be $H_{\Q}(r):=\max(\log |m|,\log |n|),$  where $r= m/n$ and $(m,n)=1$.  The ABC conjecture may thus be reformulated as follows:

\begin{conjecture}[ABC]\label{conj:ABC}
Fix $\epsilon > 0$ and suppose $u+v = 1$, where $u,v \in \Q$.  Then there exists a constant $m_{\epsilon}$ such that

\begin{equation*}
\max(H_{\Q}(u),H_{\Q}(v)) \leq m_{\epsilon}+(1+\epsilon)\sum_{p|abc}\log p,
\end{equation*}
where $u = a/c$ and $v = b/c$, and where $(a,b,c)=1$.
\end{conjecture}

An effective version of Conjecture \ref{conj:ABC} would provide an immediate means by which to solve equations of the form $u_{1}+u_{2} = 1$, where $u_{1},u_{2} \in \Q$ are $S$-units, for any finite fixed set of primes, $S$.  More generally, an effective version of the $ABC$ conjecture formulated over $K$, where $K$ denotes either a number field or a function field, would enable an effective means by which to compute all solutions to $\eqref{eq:S_unit}$, and thereby solve the Thue equation \eqref{eq:Thue_equation}.

While such a result is currently far out of reach in the classical setting, over function fields the corresponding \textit{ABC Theorem} is true, unconditionally.  In this setting, the appropriate constant $m_{\epsilon}$ may moreover be explicitly computed in terms of $g_{K}$, the \textit{genus} of $K$.  The ABC theorem may thus be used to obtain an effective upper bound for the height of any pair of $S$-units $(u_{1},u_{2}) \in K^{2}$ satisfying \eqref{eq:S_unit}.  As noted in \cite[p.~18]{Mason1984}, the bounds this method produces in the function field setting are comparatively much smaller to those obtained in the classical setting via Baker's method.

\subsection{Structure of Paper}
The remainder of this paper is structured as follows.  Section~\ref{sec:background} provides general background on valuation theory, the ABC Theorem, and discriminants, within the $\C(T)$ setting.  Section \ref{sec:S-Units} establishes certain properties of the forms $F_{\lambda}(X,Y)$ in \eqref{eq:FF_Forms}, as well as the roots $\alpha$ of the polynomial $f_{\lambda}(X) := F_{\lambda}(X,1)$.  Since a solution $(x,y) \in S_{\lambda, \xi}$ corresponds to a unit $x - \alpha y$ in the ring $\C[T][\alpha]$, in Section \ref{sec:unit_structure} we then identify a system of fundamental units for the $\C[T][\alpha]$.  In Section \ref{sec:applying_ABC} we then estimate the genus of $K$, the splitting field of $f_{\lambda}(X)$ over $\C(T)$, and apply the ABC Theorem to obtain a bound on the height of solutions to the corresponding $S$-unit equations.   Finally in \ref{sec:bound_beta_height} we apply these bounds to prove Theorem \ref{thm:main}, where the relevant computational details are then provided in the \nameref{sec:appendix}.
\subsection{Acknowledgements} The authors would like to thank Paul Voutier for suggesting this problem.  Vukusic was funded by the Austrian Science Fund (FWF) under the project I4406, as well as by the Austrian Marshall Plan Foundation with a
Marshall Plan Scholarship.  Waxman was supported by the Czech Science Foundation (GA\v{C}R) grant 17-04703Y, by a Minerva Post-Doctoral Fellowship at the Technische Universit\"at Dresden, and by a Zuckerman Post-Doctoral Fellowship at the University of Haifa.  Ziegler was funded by the Austrian Science Fund (FWF) under the project I4406.  The researchers would also like to thank AIMS Senegal, AIMS Ghana, and AIMS Rwanda for supporting research visits by Faye and Waxman, as well as the University of Salzburg for supporting a visit by Waxman.

\section{Background: Valuations, the ABC Theorem, and Discriminants}\label{sec:background}

\subsection{Valuations on $\C(T)$}\label{sec:valuations}
Let $F$ denote a field.  Recall that $v: F \rightarrow \R \cup \{\infty\}$ is said to be a \textit{valuation} on $F$ if the following properties hold (see e.g. \cite[p. 19]{Cohn1991}):
\begin{enumerate}[label=\textit{\roman*})]
\item $v(a)=\infty$ if and only if $a=0$
\item $v(ab) = v(a)+v(b)$
\item $v(a+b)\geq \min \{v(a),v(b)\}$, and\\ $v(a+b) = \min \{v(a),v(b)\}$\label{property3} whenever $v(a)\neq v(b)$.
\end{enumerate}
We say that two valuations $v_{1}$ and $v_{2}$ are \textit{equivalent} if there exists a constant $c> 0$ such that $v_{1}(f) = c\cdot v_{2}(f)$ for all $f \in F$.  A \textit{place} on $F$ is then an equivalence class of (non-trivial) valuations on $F$.  We denote the set of places on a field $F$ by $M_{F}$.  By abuse of notation we allow $v$ to refer to both a valuation and to its corresponding place.

For $a \in \C$, consider the (discrete) valuation $v_{a}:\C(T) \rightarrow \Z \cup \{\infty\}$ obtained by setting $v_{a}(T-a)=1$.  We moreover consider the \textit{valuation at infinity},
denoted $v_{\infty}$, obtained by setting $v_{\infty}(f) = -\deg(f)$ for any $f \in \C[T]$.  By an analogue of Ostrowski's theorem, we find that $M_{\C(T)} = \{v_{a}:a \in \C \cup \{\infty\}\}$.

A valuation $v$ naturally determines a \textit{norm} via $|a|_{v}:=e^{-v(a)}$.  This in turn induces a metric on $F$, whose completion we denote by $F_{v}$.  Thus, we may naturally extended $v$ to a function $v:F_{v}  \rightarrow \R \cup \{\infty\}$.  Note that the completion of $\C(T)$ with respect to $v_{\infty}$ is the field of formal Laurent series in the variable $1/T$, namely
\[\C((1/T)):=\left\{\sum_{n \geq n_{0}}a_{n}T^{-n}: n_{0} \in \Z, a_{i} \in \C, a_{n_{0}}\neq 0\right\} \cup \{0\}.\]
For any $z = \sum_{n \geq n_{0}}a_{n}T^{-n}\in \C((1/T))$ as above, we then find that $v_{\infty}(z) = n_{0}$.

Let $K/\C(T)$ denote a finite algebraic extension of degree $n$, and let $\OO_K \subseteq K$ denote the integral closure of $\C[T]$ in $K$.  To any prime ideal $\p \subseteq \mathcal{O}_{K}$ one may associate a valuation on $K$ as follows.  For any $f \in K$, we consider the principal (fractional) ideal 
\[(f)= \prod_{\p} \p^{w_{\p}(f)}.\]
Then the map $w_{\p}: f \mapsto w_{\p}(f)$ defines a valuation on $K$.

For $a \in \C$, let $(T-a)\mathcal{O}_{K}$ denote the principal ideal in $\mathcal{O}_{K}$ generated by $(T-a$), and write $(T-a)\mathcal{O}_{K} = \p_{1}^{e_{1}}\cdots \p_{g}^{e_{g}}$, where $\p_{1},\dots, \p_{g} \subseteq \mathcal{O}_{K}$ denote prime ideals.  The scaled valuation $w'_{\p_{i}} = \frac{1}{e_{i}}w_{\p_{i}}$ extends $v_{a}$ to a valuation on $K$, and we say that the place $w_{\p_{i}}$ \textit{lies above} the place $v_{a}$.  Any place $w \in M_{K}$ lying above $v_{a}$, where $a \in \C$, is referred to as a \textit{finite place} on $K$.

When $a=\infty$, we instead consider the ring $\C[1/T]$, and let $\mathcal{O}'_{K}$
denote its integral closure in $K$.  As above, we may factor $\frac{1}{T}\mathcal{O}'_{K} = \p_{1}^{e_{1}}\cdots\p_{g}^{e_{g}}$ into prime ideals in $\mathcal{O}'_{K}$.  Each such prime ideal $\p_{i}$ corresponds to a place $w_{i}\in M_{K}$  which extends $v_{\infty}$ to a valuation on $K$ (up to scaling). We say that the places $w_{1},\dots,w_{g}$ \textit{lie above} $v_{\infty}$ and refer to these as the  \textit{infinite places} on $K$.  Every place $w \in M_{K}$ is found to lie above $v_{a}$ for some $a \in \C \cup \{\infty\}$.

Each $e_{i} \in \N$ above is referred to as the \textit{ramification index} of the corresponding prime $\p_{i}$.   The prime $(T-a)\C[T]$ (resp.\ the prime $\frac{1}{T}\C[1/T]$) is said to \textit{ramify} in $K$ whenever $e_{i} > 1$ for some $i$.  We moreover find that $e_{1}+\dots + e_{g} = n$, and in the particular case that $K/\C(T)$ is Galois, we have that $e:=e_{1}=\dots = e_{g}$, i.e. that $eg=n$.

The \textit{product formula} states that
\[
	\sum_{w\in M_K} w(f) = 0 
	\quad \text{for any } f\in K.
\]
In particular, if $\mu \in \mathcal{O}_{K}^{\times}$ is a unit, then $w(\mu)=0$ at any finite place $w \in M_{K}$, from which it follows that
\begin{equation}\label{eq:unit_product_formula}
 \sum_{w| v_{\infty}} w(\mu) = 0  \quad \text{for any } \mu \in \mathcal{O}_{K}^{\times}.
\end{equation}
We moreover find that $w(\mu)=0$ at all $w \in M_{K}$ if and only if $\mu \in \C^{\times}$.

\subsection{The $\C(T)$ ABC Theorem}
Let $K$ denote a finite algebraic extension of $\C(T)$.  Recall that the \textit{height} of an element $f\in K^{\times}$ is defined to be
\[
H_{K}(f):=-\sum_{w\in M_K}\min(0,w(f)).
\]

The following theorem, a slight variation of \cite[Ch. 1 Lemma 2]{Mason1984}, provides an explicit upper bound for the height of solutions to an $S$-unit equation.  It may be viewed as a special case of the ABC-theorem for function fields:

\begin{knownthm}[ABC]\label{thm:ABC}
Let $\gamma_1, \gamma_2 \in K$ with $\gamma_1+\gamma_2=1.$  Let $\mathcal{W}$ be a finite set of valuations such that for all $w\notin \mathcal{W}$ we have $w(\gamma_1)=w(\gamma_2) = 0$. Then 
\[
	H_{K}(\gamma_1)
	\leq \max(0,2g_K-2+|\mathcal{W}|),
\]
where $g_K$ is the genus of $K$.
\end{knownthm}

The ABC Theorem is stated in terms of the genus, $g_{K}$.  A bound on $g_{K}$ may be obtained using the \textit{Riemann--Hurwitz Formula} (see e.g.\ \cite[Theorem 7.16]{Rosen2002}), which we state in the following special case:

\begin{knownthm}[Riemann--Hurwitz]\label{thm:Riemann-Hurwitz}
Let $K$ denote a finite algebraic extension of $\C(T)$. Then
\[
	2 g_K - 2
	= [K : \C(T)] \cdot (- 2) + \sum_{w\in M_K}(e_w-1),
\]
where $e_w$ denotes the ramification index of $w \in M_K$.
\end{knownthm}

\subsection{Discriminants}
Consider a principal ideal domain $A$ with field of fractions $F$.  We now recall several different notions of the \emph{discriminant}.

\begin{definition}
Let $f(X)\in F[X]$ be a monic polynomial of degree $n$, and suppose $f(X)=(X-\alpha_1)\cdots (X-\alpha_n)$, where $\alpha_{1},\dots,\alpha_{n} \in \overline{F}$, the algebraic closure of $F$.  We define the \textit{discriminant} of $f$ to be
\[
	\disc(f):= \prod_{i<j} (\alpha_i - \alpha_j)^2.
\]
\end{definition}

For $A$ and $F$ as above, let $K/F$ denote a finite Galois extension of degree $n$.  Let $\sigma_1, \ldots, \sigma_n$ moreover denote the distinct elements of the Galois group, where we note that $|\textnormal{Gal}(K/F)| =n$, since $K/F$ is Galois.

\begin{definition}
For any $e_1, \ldots, e_n\in K$ we define the \textit{discriminant} of $(e_1, \ldots, e_n)$ to be
\[
	\disc(e_1, \ldots, e_n)	:= (\det (\sigma_i(e_j))_{i,j})^2.
\]
\end{definition}

Since $K/F$ is finite and Galois, it is, in particular, finite and separable, and thus by the primitive element theorem we may write $K = F(\alpha)$, for some $\alpha \in K$.  Let $f \in F[X]$ denote the minimal polynomial of $\alpha$, and write $f(X)=(X-\alpha_{1})\cdots (X-\alpha_{n})$.  Since $K/F$ is Galois, every irreducible polynomial $f \in F[X]$ with a root in $K$ splits over $K$ and is separable.  It follows that $\alpha_{1},\dots, \alpha_{n}$ all lie in $K$ and are distinct.

For each $\sigma \in \textnormal{Gal}(K/F),$ we find that $f(\sigma(\alpha))= \sigma(f(\alpha))=0,$ and therefore $\sigma(\alpha)$ is also a root of $f(X)$.  
Note that every $\sigma$ is determined uniquely by the value of $\sigma(\alpha)$, and thus $\sigma_{i}(\alpha) \neq \sigma_{j}(\alpha)$ for $i \neq j$.
Since $|\textnormal{Gal}(K/F)| = [K:F]=\textnormal{deg}(f)=n$, we may in fact write $\sigma_{i}(\alpha):=\alpha_{i}$ for each $1 \leq i \leq n$.  We thus obtain the following relation:
\begin{align}\label{eq:disc as polynomial}
\begin{split}
\disc(1,\alpha\bb \alpha^{n-1}) = (\det (\sigma_i(\alpha^{j-1}))_{i,j})^2 &=  \prod_{i<j} (\sigma_{i}(\alpha) - \sigma_{j}(\alpha))^2\\
&= \prod_{i<j} (\alpha_i - \alpha_j)^2 = \disc(f).
\end{split}
\end{align}
Here we use the fact that $(\sigma_i(\alpha^{j-1}))_{i,j}=(\sigma_i(\alpha)^{j-1})_{i,j}$ is a Vandermonde matrix, and thus its determinant is equal to $\prod_{i<j} (\sigma_{i}(\alpha) - \sigma_{j}(\alpha))$.

Let $B$ denote the integral closure of $A$ in $K$, and let $e_1\bb e_n \in B$ denote a basis for $K/F$.

\begin{definition}
Consider the free $A$-module
\[M = \left\{\sum_{i = 1}^{n}a_{i}e_{i}: a_{i} \in A\right\} \subseteq B.\]
We define the \textit{discriminant} of $M$, denoted $D_{A}(M)$, to be the principal ideal in $A$ that is generated by $\disc(e_1\bb e_n)$.  The discriminant of the field extension $K/F$ is defined to be
\[
	D_{K/F}:=D_{A}(B).
\]
\end{definition}
Note that, indeed, $\disc(e_1\bb e_n) \in A$, and moreover that $D_{A}(M)$ is well-defined, i.e.\ does not depend on our particular choice $\{e_1\bb e_n\}$ for a basis of $M$.

\begin{knownlem}\label{lem:disc_divides}
Suppose $M'$ be an $A$-submodule of $M$ of the above form.  Then $D_{A}(M)|D_{A}(M')$, i.e.\ $D_{A}(M') \subseteq D_{A}(M)$.
\end{knownlem}
\begin{proof}
Note that $D_{A}(M')$ is generated by some $\disc(e_1'\bb e_n'),$ where $e_1'\bb e_n' \in M' \subseteq M$.  In particular, we may write $(e_1' \bb e_n')=(e_1\bb e_n)\cdot P$ for some $P\in A^{n\times n}$.  Thus $\disc(e_1'\bb e_n')=(\det P)^2 \disc(e_1\bb e_n) \in D_{A}(M)$, and therefore $D_{A}(M')\subseteq D_{A}(M)$, as desired.
\end{proof}

In subsequent computations we will make use of the following important fact about discriminants. For a proof (in a more general setting) see e.g.\ \cite[Chapter III, Corollary 2.12]{Neukirch1999}.

\begin{knownlem}\label{lem:ramifiedPrimesDiscriminant}
A prime $\p \subset A$ is ramified in $B$ if and only if $\p$ divides $D_{K/F}$.
\end{knownlem}

\section{A simple quartic family over $\C(T)$}\label{sec:S-Units}
Consider the family of quartic, binary forms
\begin{equation*}
F_\lambda(X,Y):=X^4-\lambda X^3Y-6X^2Y^{2}+ \lambda X Y^{3} +Y^{4},
\end{equation*}
where $\lambda\in \C[T]/ \{\C\}$, and let $\a:= \deg \lambda > 0 $.  Define
\begin{equation*}
f_\lambda(X):=F_{\lambda}(X,1)= X^4-\lambda X^3-6X^2+ \lambda X +1,
\end{equation*}
and note that
\[F_{\lambda}(X,Y) = Y^{4}f_{\lambda}\left(\frac{X}{Y}\right).\]

Let $\overline{\C(T)}$ denote the algebraic closure of $\C(T)$.  For $z \in \overline{\C(T)}\setminus\{0, \pm 1\}$, consider the rational maps

\begin{equation}\label{def of phi}
\phi(z) :=\frac{z-1}{z+1} \hspace{5mm} \phi^{2}(z) = -\frac{1}{z} \hspace{5mm}\phi^{3}(z) = \frac{1+z}{1-z}\hspace{5mm} \phi^{4}(z) = z,
\end{equation}
and note that $z, \phi(z),\phi^{2}(z),\phi^{3}(z)$ are distinct whenever $z \neq \pm i$.  Furthermore, if $\alpha$ is a root of $f_\lambda$, one may check that $f_\lambda\left(\phi(\alpha)\right)=0$, i.e. $\phi(\alpha)$ is also a root of $f_\lambda$.  The four distinct roots of $f_\lambda$ are thus given by $\alpha_{j}:=\phi^{j-1}(\alpha)$ for each $1 \leq j \leq 4$ (upon noting that $\alpha \neq \pm i$).

\begin{lemma}\label{lem:irreducibility_of_f}
Suppose $\deg \lambda > 0$.  Then $f_\lambda(X)$ is irreducible over $\C[T][X]$.
\end{lemma}

\begin{proof}
Suppose $f_\lambda(X) \in \C[T][X]$ is reducible.  Then either $f_\lambda(X)$ contains a root $\alpha(T) \in \C[T]$, or $f_\lambda(X)$ factors into two quadratic polynomials.  In the first case, we write $f_\lambda(X) = (X-\alpha(T))(X^{3}+a(T)X^{2}+b(T)X+c(T)),$ where $a(T),b(T),c(T) \in \C[T]$.  In particular, we have $\alpha(T)c(T)=1,$ which implies $\alpha :=\alpha(T) \in \C[T]^{\times} = \C^{\times}.$  It moreover follows from \eqref{def of phi} that $\phi(\alpha), \phi^{2}(\alpha),\phi^{3}(\alpha) \in \C$.  Thus all coefficients $f_{\lambda}$ lie in $\C$. In particular, $\lambda \in \C$, contradicting our initial assumption that
$\deg \lambda > 0$.

In the second case, we write $f_\lambda(X) = (X^{2}+a(T)X+b(T))(X^{2}+c(T)X+d(T)),$ where $a(T),b(T),c(T),d(T) \in \C[T]$.  In particular, we find that $b(T)d(T)=1$, which implies that $b(T),d(T)\in \C[T]^{\times} = \C^{\times}$.  In other words, $f_\lambda(X) = (X^{2}+a(T)X+b)(X^{2}+c(T)X+d),$ where $b, d \in \C^{\times}$.  Equating coefficients of $X^{2}$, we then find that $-6 = a(T)c(T)+b+d$, which again implies $a(T), c(T) \in \C$.  Since all coefficients $f_{\lambda}$ lie in $\C$, it follows, in particular, that $\lambda \in \C$, contradicting our initial assumption.

\end{proof}
Since $\alpha_{i} = \phi^{i-1}(\alpha) \in \C(T)(\alpha)$ for all $1 \leq i \leq 4$, we find that $K:=\C(T)(\alpha)$ is the splitting field of $f_\lambda$ over $\C(T)$.  In other words, $K$ is a normal extension, which implies $K$ is Galois.  For $\sigma \in \textnormal{Gal}(K/\C(T))$, we moreover note that $f_{\lambda}(\sigma(\alpha)) = \sigma(f_{\lambda}(\alpha))=0$, and therefore $\sigma(\alpha) = \phi^{i}(\alpha)$ for some $1 \leq i \leq 4$. By Lemma~\ref{lem:irreducibility_of_f}, $|\textnormal{Gal}(K/\C(T))| =\deg(f_{\lambda})=4$.  Since $\sigma$ is uniquely determined by the value of $\sigma(\alpha) \in K$, we can define each $\sigma_{1},\sigma_{2},\sigma_{3},\sigma_{4} \in \textnormal{Gal}(K/\C(T))$ by setting $\sigma_{i}(\alpha)=\alpha_{i}$.


Let $(x,y)\in\C[T]\times \C[T]$ denote some solution to $F_\lambda(X,Y)=\xi$, where $\xi\in \C^{\times}$.  Define 

\begin{equation*}
\beta_i:=x-\alpha_i y
\end{equation*}
and write $\beta:=\beta_1=x-\alpha y.$
Since
\begin{align*}
	F_\lambda(x,y) &= y^{4}f_{\lambda}\left(\frac{x}{y}\right) = y^{4}(x/y-\alpha_1)(x/y-\alpha_2)(x/y-\alpha_3)(x/y-\alpha_4)\\
	&=(x-\alpha_1 y)(x-\alpha_2 y)(x-\alpha_3 y)(x-\alpha_4 y)=\xi,
\end{align*}
the elements $\beta_i=x-y\alpha_i$ are  units in the ring $\C[T][\alpha_1,\alpha_2,\alpha_3,\alpha_4]$.  Conversely, any unit $\beta \in \C[T][\alpha_1,\alpha_2,\alpha_3,\alpha_4]$ of the form $\beta=x - \alpha y$ yields a solution $(x,y)\in S_{\lambda, \xi}$, for some $\xi \in \C^{\times}$.  Thus, finding the solution set $S_{\lambda, \xi}$ for all $\xi \in \C^{\times}$ is equivalent to finding the set of units $\beta \in \C[T][\alpha_1,\alpha_2,\alpha_3,\alpha_4]^{\times}$ of the shape $\beta=x - \alpha y,$ where $x,y \in \C[T]$.  To better understand such units, we begin by noting the following lemma.

\begin{lemma}
Let $\alpha_1,\alpha_2,\alpha_3,\alpha_4$ denote the roots of $f_\lambda(X)$. Then $\C[T][\alpha_1,\alpha_2,\alpha_3,\alpha_4]=\C[T][\alpha_{1}]$.
\end{lemma}
\begin{proof}
It suffices to demonstrate that $\alpha_2,\alpha_3,\alpha_4 \in \C[T][\alpha]=\{A\alpha^3 +B\alpha^2+C\alpha+D: A,B,C,D\in \C[T]\}$, where $\alpha:=\alpha_{1}$.  To show that $\alpha_2 \in \C[T][\alpha]$, we note that $\alpha_2=\phi(\alpha)=(\alpha - 1)/(\alpha +1 )$. Since clearly $\alpha - 1 \in \C[T][\alpha]$, it suffices to demonstrate that $(\alpha+1)^{-1} \in \C[T][\alpha]$.  Let us write
\[(\alpha+1)^{-1} = A\alpha^3 +B\alpha^2+C\alpha+D, \hspace{5mm} A, B, C,D \in \C(T),\]
and note that 
$(\alpha+1)^{-1} \in \C[T][\alpha]$ if and only if $A, B, C,D \in \C[T]$.  We then compute

\begin{align*}
	1 &= (\alpha + 1)(A\alpha^3 +B\alpha^2+C\alpha+D)\\
	&=A \alpha^4 + (A+B)\alpha^3 + (B+C)\alpha^2 + (C+D)\alpha + D\\
	&= A(\lambda \alpha^3 + 6 \alpha^2 - \lambda\alpha -1) + (A+B)\alpha^3 + (B+C)\alpha^2 + (C+D)\alpha + D\\
	&= (A \lambda + A + B) \alpha^3 + (6A + B+C)\alpha^2 + (-\lambda A + C+D)\alpha + (-A+D).
\end{align*}
Comparing coefficients and solving the system of equations
\begin{align*}
	A(\lambda +1) + B = 0, \quad
	6A+B+C = 0, \quad
	-\lambda A + C + D = 0, \quad
	-A + D = 1,
\end{align*}
we get that 
\[
	A= \frac{1}{4}, \quad
	B=\frac{-\lambda-1}{4}, \quad
	C=\frac{\lambda-5}{4},\quad
	D=\frac{5}{4}.
\]
It follows that
\begin{equation}\label{eq:1overAlpha+1}
	\frac{1}{(\alpha+1)} = \frac{1}{4} \left(\alpha^3 -(\lambda+1)\alpha^2 + (\lambda-5)\alpha + 5\right).
\end{equation}

Thus, $\alpha_2 = (\alpha -1)/(\alpha+1)\in \C[T][\alpha]$, and therefore $\C[T][\alpha_{2}]\subseteq \C[T][\alpha]$.  By the exact same argument, we find that 
$\C[T][\alpha_{3}]\subseteq \C[T][\alpha_{2}]$, and also that $\C[T][\alpha_{4}]\subseteq \C[T][\alpha_{3}]$, i.e.\ that $\C[T][\alpha_2,\alpha_3,\alpha_4] \subseteq \C[T][\alpha]$, from
 which the claim then follows.
\end{proof}
\subsection{Computing Laurent Series of $\alpha$} The following is a corollary of Hensel's Lemma:

\begin{knownlem}\label{lem:HenselCor}
If $f(t,X)$ is a polynomial in two variables over a field $k$, and $X=a$ is a simple root of $f(0,X)$, then there is a unique power series $X(t)$ with $X(0)=a$ and $f(t,X(t))=0$ identically.
\end{knownlem}
\begin{proof}
See \cite[Corollary 7.4]{Eisenbud1995}.
\end{proof}

\begin{lemma}\label{lem:alpha_shape}
The polynomial $f_{\lambda}(X) = X^4 - \lambda X^3 - 6 X^2 + \lambda X + 1$ has four distinct roots in $\C((1/\lambda))$, which take the following shape:
\begin{align*}
\alpha &= 1 - \frac{2}{\lambda} + \frac{2}{\lambda^2} + \frac{8}{\lambda^3} + \dots &\alpha_{2}&=-\frac{1}{\lambda} + \frac{5}{\lambda^3} + \dots \\
\alpha_{3}&=-1 -\frac{2}{\lambda}-\frac{2}{\lambda^2} +\frac{8}{\lambda^3} +\dots &
\alpha_{4}&= \lambda + \frac{5}{\lambda} + \dots.
\end{align*}
\end{lemma}
\begin{proof}
Note that $f_\lambda(\alpha) = 0$ if and only if $\tilde{f}(1/\lambda,\alpha) = 0$, where
\begin{equation*}
	\tilde{f}\left(\frac{1}{\lambda},X\right)
	:=\frac{1}{\lambda} f_\lambda(X)	
	= \frac{1}{\lambda} X^4 - X^3 - \frac{6}{\lambda} X^2 + X + \frac{1}{\lambda}
	=0.
\end{equation*}
Note further that $-1,0,1$ are each simple roots of $\tilde{f}(0,X)=-X^3+X$.  In particular, $1$ is a simple root of $\tilde{f}(0,X)$.  By Lemma \ref{lem:HenselCor}, there then exists a unique power series of the form $X(1/\lambda)=1 + a_1/\lambda + a_2/\lambda^2 + \dots$, such that

\[\tilde{f}\left(\frac{1}{\lambda},X\left(\frac{1}{\lambda}\right)\right)=0.\]
Equivalently, $X(1/\lambda)$ is a root of $f_\lambda(X)$. Let us call this root $\alpha$, i.e.
\[\alpha = 1 + \frac{a_1}{\lambda} + \frac{a_2}{\lambda^2} + \dots.\]
In order to explicitly compute the coefficients of this expansion, we note that
\begin{align*}
	\frac{1}{\lambda} \left(1 + a_1 \frac{1}{\lambda} + \dots\right)^4 
	&- \left(1 + a_1 \frac{1}{\lambda} + \dots\right)^3 \\
	&- \frac{6}{\lambda} \left(1 + a_1 \frac{1}{\lambda}  + \dots\right)^2 
	+ \left(1 + a_1 \frac{1}{\lambda} + \dots\right) 
	+ \frac{1}{\lambda}
	=0,
\end{align*}
and compare coefficients.  The coefficient of $1/\lambda$ on the left-hand side is equal to $1 -3a_1 -6 +a_1+1$, which upon setting equal to 0, implies $a_1=-2$.  Considering higher powers of $1/\lambda$, we similarly find that
\begin{equation*}
	\alpha = 1 - \frac{2}{\lambda} + \frac{2}{\lambda^2} + \frac{8}{\lambda^3} + \dots.
\end{equation*}

To obtain the Laurent series representations for the other roots of $f_\lambda(X)$, we recall that $1/(1-x)=1+x+x^2+\dots$, and then compute
\begin{align*}
	\alpha_2 
	&= \phi(\alpha)=\frac{\alpha-1}{\alpha+1}
	= \frac{- \frac{2}{\lambda} + \frac{2}{\lambda^2} + \frac{8}{\lambda^3} + \dots}{2 - \frac{2}{\lambda} + \frac{2}{\lambda^2} + \frac{8}{\lambda^3} + \dots}
	= \frac{-\frac{1}{\lambda} + \frac{1}{\lambda^2} + \frac{4}{\lambda^3} + \dots}{1 - \frac{1}{\lambda} + \frac{1}{\lambda^2} + \frac{4}{\lambda^3} + \dots}\\
	&= \left(-\frac{1}{\lambda} + \frac{1}{\lambda^2} + \frac{4}{\lambda^3} + \dots \right) 
	\frac{1}{1-(\frac{1}{\lambda} - \frac{1}{\lambda^2} - \frac{4}{\lambda^3} + \dots)}\\
	&= \left(-\frac{1}{\lambda} + \frac{1}{\lambda^2} + \frac{4}{\lambda^3} + \dots \right) 
	\left( 1 + (\frac{1}{\lambda} - \frac{1}{\lambda^2} - \frac{4}{\lambda^3} + \dots) + (\frac{1}{\lambda} - \frac{1}{\lambda^2} - \frac{4}{\lambda^3} + \dots)^2 + \dots \right)\\
	&= \left(-\frac{1}{\lambda} + \frac{1}{\lambda^2} + \frac{4}{\lambda^3} + \dots \right) 
	\left( 1+ \frac{1}{\lambda} - \frac{5}{\lambda^3} + \dots \right)
	= -\frac{1}{\lambda} + \frac{5}{\lambda^3} + \dots. 	
\end{align*}
The roots $\alpha_{3} = 1/\alpha$ and $\alpha_{4} = -1/\alpha_{2}$ may then be computed similarly.
\end{proof}

Above we explicitly computed the four distinct roots of $f_{\lambda}$ in $\C((1/\lambda))$.  Note that $\C((1/\lambda))$ embeds into $\C((1/T))$, since $\lambda=\lambda_\a T^\a + \dots + \lambda_0$ lies in $\C((1/T))$ and $|1/\lambda|_{v_{\infty}} < 1$.  Thus $f_{\lambda}$ has four distinct roots in $\C((1/T))$, each of which corresponds to a unique embedding $\iota: K \hookrightarrow \C((1/T))$ defined by $\iota_{i}: \alpha \rightarrow \alpha_{i}$ for some $1 \leq i \leq 4$.  Each embedding then induces a valuation $w_{i}: K \rightarrow \Z \cup \{\infty\}$ given by $w_{i}(z) = v_{\infty}(\iota_{i}(z))$ for all $z \in K$.  In particular, each $w_{i}$ extends the valuation $v_{\infty}$ on $\C(T)$, and we will see from the computations below that $w_{1}, w_{2}, w_{3},$ and $w_{4}$ are distinct, i.e.\  that $v_{\infty}$ does not ramify over $K$.

For $z \in K$, we moreover define
\[(z)_{\infty}:=(w_{1}(z),w_{2}(z),w_{3}(z),w_{4}(z)).\]

For any $z \in K$, let $z_{i}:=\sigma_{i}(z)$ for $1 \leq i \leq 4$ denote the conjugates of $z$. Considering $i+j-1$ mod 4, we note that
\[\iota_{j}(\sigma_{i}(\alpha)) = \iota_{j}(\phi^{i-1}(\alpha)) = \phi^{i-1}(\iota_{j}(\alpha))= \phi^{i-1}(\alpha_{j}) = \alpha_{i+j-1}= \iota_{i+j-1}(\alpha),\]
and therefore that in fact $\iota_{j}(\sigma_{i}(z)) = \iota_{i+j-1}(z)$ for all $z \in K$.  We thus find that
\begin{equation*}
w_{j}(z_{i})= v_{\infty}(\iota_{j}(z_{i}))=v_{\infty}(\iota_{j}(\sigma_{i}(z)))=v_{\infty}(\iota_{i+j-1}(z))=w_{i+j-1}(z),
\end{equation*}
and conclude that, for any $i,j \in \{1,2,3,4\}$, the following sets are equal:
\begin{align}\label{eq:sets equal}
\begin{split}
\{w_{1}(z),w_{2}(z),w_{3}(z),w_{4}(z)\}
&=\{w_1(z_i),w_2(z_i),w_3(z_i),w_4(z_i)\}\\
&= \{w_{j}(z_{1}),w_{j}(z_{2}),w_{j}(z_{3}),w_{j}(z_{4})\}.
\end{split}
\end{align}

\section{Unit Structure of $\C[T][\alpha]^\times$}\label{sec:unit_structure}

Next, we wish to find a \textit{system of fundamental units} for $\C[T][\alpha]$.  Note that since $\alpha\alpha_2\alpha_3\alpha_4=1$,
we find, in particular, that $\alpha$ is a unit in $\C[T][\alpha]$.  Similarly, from \eqref{eq:1overAlpha+1} we know that $\alpha+1$ is a unit in $\C[T][\alpha]$. Finally, as $\alpha_{2}$ is a unit, it follows that $\alpha - 1= \alpha_2 (1+\alpha)$ is also a unit.  We wish to show that $\alpha, \alpha + 1,$ and $\alpha -1$ form a fundamental system for $\C[T][\alpha]^{\times}$.  To this end, we proceed by computing the valuations of $\alpha, \alpha+1$, and $\alpha - 1$ at the four places lying above $v_{\infty}$.
\begin{lemma}\label{lem:val_of_fundamental_units}
We have the following valuations:
\begin{align*}
(\alpha)_{\infty} = (0,\a,0,-\a),\quad (\alpha-1)_\infty =(\a,0,0,-\a), \quad 
(\alpha+1)_\infty =(0,0,\a,-\a).
\end{align*}
\end{lemma}
\begin{proof}
Since $v_{\infty}(c/\lambda^{n}) = n\a$ for any $c \in \C^{\times}$, it follows from Lemma \ref{lem:alpha_shape} that
\begin{align*}
w_{1}(\alpha) = v_{\infty}(\alpha_{1}) &= v_{\infty}\left(1 - \frac{2}{\lambda} + \frac{2}{\lambda^2} + \frac{8}{\lambda^3} + \dots\right) = v_{\infty}(1) = 0,
\end{align*}
and similarly that
\begin{align*}
w_{2}(\alpha)&=v_\infty(\alpha_2)= 
v_\infty\left(-\frac{1}{\lambda} + \frac{5}{\lambda^3} + \dots\right)=\a\\
w_{3}(\alpha)&=	v_\infty(\alpha_3)
	=v_\infty\left(-1 -\frac{2}{\lambda}-\frac{2}{\lambda^2} +\frac{8}{\lambda^3} +\dots\right) = 0\\
w_{4}(\alpha)&= v_\infty(\alpha_4)
	= v_\infty\left(\lambda + \frac{5}{\lambda} + \dots\right) = -\a.
\end{align*}

from which it follows that $(\alpha)_{\infty} = (0,\a,0,-\a)$.  Moreover,
\begin{equation*}
\begin{aligned}[c]
	\alpha_1 - 1 
	&= - \frac{2}{\lambda} + \frac{2}{\lambda^2} + \frac{8}{\lambda^3} + \dots,\\
	\alpha_2 -1
	&= -1 -\frac{1}{\lambda} + \frac{5}{\lambda^3} + \dots,\\
	\alpha_3 -1
	&=-2 -\frac{2}{\lambda}-\frac{2}{\lambda^2} +\frac{8}{\lambda^3} +\dots,\\
	\alpha_4 -1
	&=\lambda -1 + \frac{5}{\lambda} + \dots.
\end{aligned}
\quad
\begin{aligned}[c]
	\alpha_1 + 1 
	&= 2 - \frac{2}{\lambda} + \frac{2}{\lambda^2} + \frac{8}{\lambda^3} + \dots,\\
	\alpha_2 +1
	&=1 -\frac{1}{\lambda} + \frac{5}{\lambda^3} + \dots,\\
	\alpha_3 +1
	&=-\frac{2}{\lambda}-\frac{2}{\lambda^2} +\frac{8}{\lambda^3} +\dots,\\
	\alpha_4 +1
	&= \lambda +1 + \frac{5}{\lambda} + \dots,
\end{aligned}
\end{equation*}
from which it follows that $(\alpha-1)_\infty=(\a,0,0,-\a)$ and $(\alpha+1)_\infty=(0,0,\a,-\a)$, as desired.

\end{proof}

By Lemma \ref{lem:val_of_fundamental_units}
we see that $(\alpha-1)_\infty$, $(\alpha)_\infty$ and $(\alpha+1)_\infty$, are linearly independent, and therefore that $\alpha$, $\alpha-1$, and $\alpha+1$ are multiplicatively independent.  In other words, for any $r,s,t \in \Z$, we find that
\[\alpha^{r}(\alpha-1)^{s}(\alpha+1)^{t}=1 \Leftrightarrow r,s,t = 0.\]
In fact, we have the following:

\begin{prop}\label{prop:fundSystem}
The units $\alpha-1$, $\alpha$ and $\alpha+1$ form a fundamental system for $\C[T][\alpha]^\times$, namely every $\eps\in \C[T][\alpha]^\times$ can be represented as 
\[
	\eps = \eta (\alpha-1)^r \alpha^s (\alpha+1)^t,
\]
with $\eta \in \C^\times$ and $r,s,t\in \Z$.
\end{prop}

In order to prove Proposition \ref{prop:fundSystem}, we first prove the following lemma.

\begin{lemma}\label{lem:min_ei}
Let $\eps \in \C[T][\alpha]^{\times}$.  Then either $\eps \in \C^{\times}$ or $\min\{e_1,e_2,e_3,e_4\}\leq -\a$, where $(\eps)_\infty:=(e_1,e_2,e_3,e_4)$.
\end{lemma}

\begin{proof}
For $\eps \in \C[T][\alpha]^\times$, let $\eps_{i}:=\sigma_{i}(\eps)$ for $1 \leq i \leq 4$ denote the conjugates of $\eps$.  Since $\eps$ is a unit, by \eqref{eq:unit_product_formula}
we find that $e_1+e_2+e_3+e_4=0$. If $e_1=e_2=e_3=e_4=0$, then $\eps \in \C^\times$ and we are done. Otherwise there exists some $e_{i_{0}}>0$.
By \eqref{eq:sets equal}, we moreover note that
\begin{align*}
\{e_{1},e_{2},e_{3},e_{4}\}&= \{w_{2}(\eps_{1}),w_{2}(\eps_{2}),w_{2}(\eps_{3}),w_{2}(\eps_{4})\},
\end{align*}
and thus there exists some $i$ such that $w_{2}(\eps_{i}) > 0$.  From \eqref{eq:sets equal} it further follows that
\begin{align*}
\{e_{1},e_{2},e_{3},e_{4}\}&=\{w_1(\eps_i),w_2(\eps_i),w_3(\eps_i),w_4(\eps_i)\}
\end{align*}
and thus we may replace $\eps$ by $\eps_{i}$ and assume, without loss of generality, that $e_{2}> 0$.

Since $\eps \in \C[T][\alpha]^\times \subset \C[T][\alpha]$, we can write
\begin{equation*}	\eps_i=h_0+h_1\alpha_i+h_2\alpha_i^2+h_3\alpha_i^3 
	\quad \text{for } i=1,2,3,4,
\end{equation*}
with $h_0,h_1,h_2,h_3\in \C[T]$.  We wish to solve this system of linear equations, and we do so using Cramer's rule, namely that
\[
	h_0 = \frac{\det A_1}{\det A},
\]
where
\[
	A 
	= \begin{pmatrix}
	1 & \alpha_1 & \alpha_1^2 & \alpha_1^3 \\
	1 & \alpha_2 & \alpha_2^2 & \alpha_2^3 \\
	1 & \alpha_3 & \alpha_3^2 & \alpha_3^3 \\
	1 & \alpha_4 & \alpha_4^2 & \alpha_4^3 \\
	\end{pmatrix}
	\quad \text{and} \quad
	A_1
	= \begin{pmatrix}
	\eps_1 & \alpha_1 & \alpha_1^2 & \alpha_1^3 \\
	\eps_2 & \alpha_2 & \alpha_2^2 & \alpha_2^3 \\
	\eps_3 & \alpha_3 & \alpha_3^2 & \alpha_3^3 \\
	\eps_4 & \alpha_4 & \alpha_4^2 & \alpha_4^3 \\
	\end{pmatrix},
\]

The matrix $A$ is a Vandermonde matrix, and therefore

\[\det A =\prod_{1\leq i < j \leq 4} (\alpha_j - \alpha_i)
	= (\alpha_4-\alpha_3)(\alpha_4-\alpha_2)(\alpha_4-\alpha_1)(\alpha_3-\alpha_2)(\alpha_3-\alpha_1)(\alpha_2-\alpha_1).\]
Hence
\begin{align*}
\iota_{1} (\det A)
	&= (\lambda + \dots)(\lambda + \dots)(\lambda + \dots)
		(-1 + \dots)(-2+\dots)(-1+\dots)
	=-2\lambda^3 + \dots,
\end{align*}
from which it follows that $w_1(\det A)=-3\a$. Since $\iota_{k}:\alpha_{i} \mapsto \alpha_{i+k-1}$, we see, moreover, that $\iota_{k}(\det A) = \pm \iota_{1}(\det A)$.  Thus $w_{k}(\det A) = w_{1}(\det A)$ for all $1 \leq k \leq 4$, and we conclude that $(\det A)_\infty  =(-3\a,-3\a,-3\a,-3\a).$

If we compute $\det A_1$, we get that
\begin{align*}
	\det A_1 
	&= \eps_1 \alpha_2 \alpha_3 \alpha_4 (\alpha_2 - \alpha_3)(\alpha_3-\alpha_4)(\alpha_4-\alpha_2)\\
	 &\phantom{=} - \eps_2 \alpha_3 \alpha_4 \alpha_1 (\alpha_3 - \alpha_4)(\alpha_4-\alpha_1)(\alpha_1-\alpha_3)\\ 
	&\phantom{=} + \eps_3 \alpha_4 \alpha_1 \alpha_2 (\alpha_4-\alpha_1)(\alpha_1-\alpha_2)(\alpha_2-\alpha_4)\\
	&\phantom{=} - \eps_4 \alpha_1 \alpha_2 \alpha_3 (\alpha_1-\alpha_2)(\alpha_2-\alpha_3)(\alpha_3-\alpha_1)\\
	&= \delta - \sigma(\delta) + \sigma^2(\delta) - \sigma^3(\delta),
\end{align*}
where
\[
	\delta = \eps_1 \alpha_2 \alpha_3 \alpha_4 (\alpha_2 - \alpha_3)(\alpha_3-\alpha_4)(\alpha_4-\alpha_2).
\]
Since $(\eps)_\infty=(e_1,e_2,e_3,e_4)$, we write $\iota_{1}(\eps_1)= c_1 T^{-e_1} + \dots$, and compute
\begin{align*}
	\iota_{1}(\delta) 
	&= (c_1T^{-e_1} + \dots) (-\frac{1}{\lambda} + \dots) (-1 + \dots)(\lambda + \dots)(1 + \dots)(-\lambda + \dots)(\lambda + \dots)\\
	&= -c_1T^{-e_1} \lambda^2 + \dots,
\end{align*}
so $w_1(\delta)=e_1-2\a$. Similarly, we compute $\iota_{2}(\delta)$, $\iota_{3}(\delta)$ and $\iota_{4}(\delta)$ to obtain $w_2(\delta)$, $w_3(\delta)$ and $w_4(\delta)$. We conclude that $(\delta)_\infty=(e_1-2\a,e_2-3\a,e_3-2\a,e_4+\a)$. 

Now for any $i=1,2,3,4$,
\begin{align*}
	w_i(\det A_1)
	=w_i(\delta - \sigma(\delta) + \sigma^2(\delta) - \sigma^3(\delta))
	&\geq \min\{w_i(\delta),w_i(\sigma(\delta)),w_i(\sigma^{2}(\delta)),w_i(\sigma^{3}(\delta))\}\\
	&= \min\{	e_1-2\a,e_2-3\a,e_3-2\a,e_4+\a\},
\end{align*}
where the last step follows from \eqref{eq:sets equal}.  Dividing by $\det A$ we obtain
\begin{align*}
	w_i(h_0)
	= w_i\left(\frac{\det A_1}{\det A}\right)
	= w_i(\det A_1) - w_i(\det A)
	&\geq \min\{e_1-2\a,e_2-3\a,e_3-2\a,e_4+\a\} + 3\a\\
	&= \min \{e_1+\a,e_2,e_3+\a,e_4+4\a\}.
\end{align*}
Recall that $h_0\in \C[T]$, and assume for the moment that $h_{0} \neq 0$.  Then $w_i(h_0)=v_{\infty}(h_{0})=-\deg h_0 \leq 0$ for $i=1,2,3,4$, so $\min \{e_1+\a,e_2,e_3+\a,e_4+4\a\}\leq 0$. Since we assume $e_2>0$, it follows that $\min\{e_1+\a,e_3+\a,e_4+4\a\}\leq 0$, which implies $\min \{e_1,e_3,e_4\}\leq -\a$.

Finally, we consider the case $h_0 = 0$, i.e. we assume that
\[
	\eps 
	= \alpha (h_1 + h_2 \alpha + h_3 \alpha^2),
\]
where $h_1, h_2, h_3 \in \C[T]$.
We consider two subcases, based on whether or not the following chain of equalities holds:
\begin{equation}\label{eq:degrees-specialCase}
	\deg h_1 
	= \deg h_2 + \a
	= \deg h_3 + 2 \a.
\end{equation}
Suppose first that \eqref{eq:degrees-specialCase} does not hold.   Then
\[
	w_4(\eps)
	= w_4(\alpha) + w_4(h_1 + h_2 \alpha + h_3 \alpha^2)
	= - \a 	+ w_4(h_1 + h_2 \alpha + h_3 \alpha^2)
	\leq - \a,
\]
and we are done. 
Note that for the last inequality we used the following two facts: 
First, for any valuation $v$ and any elements $a,b,c$ we have $v(a+b+c)\leq \max\{v(a),v(b),b(c)\}$ so long as $v(a),v(b),v(c)$ are not all equal. Second, $w_4(h_1) = -\deg h_1$, $w_4(h_2 \alpha) = - \deg h_2 - \a$, $w_4(h_3 \alpha^2) = - \deg h_3 -2 \a$ are each $\leq 0$
and the three numbers are not all equal, since we are assuming that \eqref{eq:degrees-specialCase} does not hold.

Suppose next that \eqref{eq:degrees-specialCase} does hold. Then
\[
	w_1(\eps)
	= w_1(\alpha) + w_1(h_1 + h_2 \alpha + h_3 \alpha^2)
	= 0 + w_1(h_1 + h_2 \alpha + h_3 \alpha^2).
\]
By \eqref{eq:degrees-specialCase} we have 
$w_1(h_1)= -\deg h_1 = -\deg h_3 - 2\a$, 
$w_1(h_2 \alpha) =  - \deg h_2 = -\deg h_3 - \a$,
$w_1(h_3 \alpha^2) = - \deg h_3$, which are all distinct. Thus we obtain
\begin{align*}
	w_1(\eps)
	&= w_1(h_1 + h_2 \alpha + h_3 \alpha^2)
	= \min \{ -\deg h_3 - 2\a, -\deg h_3 - \a, - \deg h_3 \}\\
	&= -\deg h_3 - 2\a
	\leq -\a,
\end{align*}
and we are done.
\end{proof}

\begin{proof}[Proof of Proposition \ref{prop:fundSystem}]
Let $\eps\in\C[T][\alpha]^\times$ be an arbitrary unit. Recall that $(\alpha-1)_\infty=(\a,0,0,-\a)$, $(\alpha)_\infty=(0,\a,0,-\a)$ and $(\alpha+1)_\infty=(0,0,\a,-\a)$. Clearly, we can multiply $\eps$ with powers of $\alpha-1,\alpha,\alpha+1$ to obtain a new unit of the form $\eps'=\eps (\alpha-1)^r \alpha^s (\alpha+1)^t$, where 
$(\eps')_{\infty}=(e_1', e_2',e_3',e_4')$ is such that 
$\a \leq e_1' < 2\a$ and $-\a<e_2',e_3'\leq 0$. Since $e_1'+e_2'+e_3'+e_4'=0$, we have $e_4'=-e_1'-e_2'-e_3'$ and therefore $e_4'>-\a$. It follows that $\min\{e_1',e_2',e_3',e_4'\}>-\a$. But then Lemma~\ref{lem:min_ei} implies that $\eps'\in \C^\times$, so
\[
	\eps = \eps'(\alpha-1)^{-r} \alpha^{-s} (\alpha+1)^{-t}, \quad \eps' \in \C^{\times},
\]
as desired.
\end{proof}

\section{Applying the ABC Theorem}\label{sec:applying_ABC}

\subsection{Computing $D_{K/\C(T)}$ and Estimating $g_{K}$}

\begin{lemma}\label{lem:bound on rk}
Let $r_{K}$ denote the number of places $v \in M_{\C(T)}$ which ramify in $K$.  Then $r_{K} \leq 2 \a$.
\end{lemma}
\begin{proof}
Since $\alpha$ is integral over $\C[T]$, we have that $\C[T][\alpha]\subseteq \OO_K$, where  $\OO_K$ denotes the integral closure of $\C[T]$ in $K$.  Upon noting that 
$\C[T][\alpha]$ is a $\C[T]$-module with basis $\{1,\alpha,\alpha^2,\alpha^3\}$, it follows from Lemma \ref{lem:disc_divides} that the discriminant $D_{K/\C(T)}$ divides the discriminant $D_{\C[T]}(\C[T][\alpha])$.  By \eqref{eq:disc as polynomial} we then compute
\[D_{\C[T]}(\C[T][\alpha])=\disc(1,\alpha,\alpha^{2},\alpha^{3})\C[T] = \disc(f_\lambda)\C[T]=4(\lambda^2 + 16)^3 \C[T].\]
By Lemma \ref{lem:ramifiedPrimesDiscriminant}, a prime $(T-a) \subset \C[T]$ can only ramify in $K$ if it divides $(\lambda^2 + 16)$, i.e.\  if $a$ is a root of $\lambda^2+16$. Since $\deg \lambda = \a$, there are at most $2\a$ such primes. Since, moreover, we have already seen that $v_{\infty}$ does not ramify, we conclude that there are at most $2\a$ primes that ramify, as desired.
\end{proof}
Now we can use the \nameref{thm:Riemann-Hurwitz} Formula to bound the genus of $K$, which will then be applied in the \nameref{thm:ABC} Theorem.

\begin{lemma}\label{lem:genusK}
Let $r_{K}$ denote the number of places in $\C(T)$ which ramify in $K$, and let $g_{K}$ denote the genus of $K$. Then
\[
	g_K 
	\leq \frac{3}{2} r_{K} - 3
	\leq 3\a - 3.
\]
\end{lemma}
\begin{proof}
Since $[K:\C(T)]=4$ and the ramification index of each ramified prime is at most 4, 
it follows from the \nameref{thm:Riemann-Hurwitz} Formula that
\begin{align*}
	2 g_K - 2
	&= [K : \C(T)] \cdot (- 2) + \sum_{w\in M_K}(e_w-1)\\
&\leq  4(-2) +  r_{K}(4-1),
\end{align*}
which implies $g_K \leq 3r_{K}/2 - 3$.  The second inequality now follows by Lemma \ref{lem:bound on rk}.
\end{proof}

\subsection{Application of the ABC Theorem} 
In what follows, we use the \nameref{conj:ABC} Theorem to first estimate the height 
$(\alpha_{2}-\alpha_{3})\beta_{1}/(\alpha_{3}-\alpha_{1})\beta_{2}$, which we in turn use to bound the height of $\beta$.

\begin{lemma}\label{lem:upperBound-H_K-gammas}
We have that
\begin{equation*}\label{gamma bound}
H_K\left(\frac{(\alpha_{2}-\alpha_{3})\beta_{1}}{(\alpha_{3}-\alpha_{1})\beta_{2}}\right)\leq 10\a -4.
\end{equation*}
\end{lemma}
\begin{proof}
By Siegel's identity,
\begin{align*}
	&\beta_1(\alpha_2-\alpha_3) +\beta_2(\alpha_3-\alpha_1)+\beta_3(\alpha_1-\alpha_2)	\\
	&=(x-\alpha_1y)(\alpha_2-\alpha_3) +(x-\alpha_2 y)(\alpha_3-\alpha_1)+(x-\alpha_3y)(\alpha_1-\alpha_2)
	=0,
\end{align*} 
which further implies that
\[-\frac{(\alpha_{2}-\alpha_{3})\beta_{1}}{(\alpha_{3}-\alpha_{1})\beta_{2}} - \frac{(\alpha_{1}-\alpha_{2})\beta_{3}}{(\alpha_{3}-\alpha_{1})\beta_{2}} =1.\]
Applying Theorem~\ref{thm:ABC}, we then obtain that
\begin{equation}\label{eq:Mason bound II}
	H_K\left(\frac{(\alpha_{2}-\alpha_{3})\beta_{1}}{(\alpha_{3}-\alpha_{1})\beta_{2}}\right)\leq \max(0,2g_K-2+|\mathcal{W}|),
\end{equation}
where $\mathcal{W}$ denotes the set of valuations $w \in M_{K}$ for which either 
\[w\left(\frac{(\alpha_{2}-\alpha_{3})\beta_{1}}{(\alpha_{3}-\alpha_{1})\beta_{2}}\right)\neq 0 \quad \textnormal{ or } \quad w\left(\frac{(\alpha_{1}-\alpha_{2})\beta_{3}}{(\alpha_{3}-\alpha_{1})\beta_{2}}\right)\neq 0.\]
\\
We bound the size of $|\mathcal{W}|$ from above, by counting the number of valuations for which either 
\begin{equation}\label{eq:non_zero_valuations}
w\left((\alpha_{2}-\alpha_{3})\beta_{1}\right)\neq 0 \quad \textnormal{or} \quad w\left((\alpha_{3}-\alpha_{1})\beta_{2}\right)\neq 0 \quad \textnormal{or} \quad w\left((\alpha_{1}-\alpha_{2})\beta_{3}\right)\neq 0.
\end{equation}
Since $(\alpha_{2}-\alpha_{3})\beta_{1}, (\alpha_{3}-\alpha_{1})\beta_{2}, (\alpha_{1}-\alpha_{2})\beta_{3}\in \OO_K$, we find that
\[w\left((\alpha_{2}-\alpha_{3})\beta_{1}\right), w\left((\alpha_{3}-\alpha_{1})\beta_{2}\right), w\left((\alpha_{1}-\alpha_{2})\beta_{3}\right)\geq 0
\]
at every finite place $w \in M_{K}$.  Hence, \eqref{eq:non_zero_valuations} holds at a given valuation $w \in M_{K}$ if and only if 
\[w\left((\alpha_2-\alpha_3)\beta_1(\alpha_3-\alpha_1) \beta_2(\alpha_1-\alpha_2) \beta_3\right) > 0.
\]
Since the $\beta_i$ are moreover units, and $\disc(f_{\lambda})=\prod_{1\leq i < j \leq 4}(\alpha_i-\alpha_j)^2$, we have that 
\[(\alpha_1-\alpha_2)(\alpha_2-\alpha_3)(\alpha_3-\alpha_1) \beta_1\beta_2 \beta_3| \disc(f_\lambda)=4(\lambda^2+16)^3.\]

Note that there are at most $2\a+1$ distinct valuations $v \in M_{\C(T)}$ such that $v(\disc(f))\neq 0$.  Therefore,
\[
	|\mathcal{W}|	
	\leq 2r_{K} + 4(2\a+1-r_{K})
	= 4 + 8\a -2r_{K}.
\]
Here we use the fact that if $v$ ramifies, then there are at most 2 distinct valuations lying above $v$, while if $v$ is unramified then there are exactly 4.

Finally, from \eqref{eq:Mason bound II} and the bound for $g_{K}$ provided in Lemma \ref{lem:genusK}, we conclude that
\[
H_K\left(\frac{(\alpha_{2}-\alpha_{3})\beta_{1}}{(\alpha_{3}-\alpha_{1})\beta_{2}}\right)
	\leq 2\left(\frac{3}{2}r_{K}-3\right) - 2+ 4 + 8\a -2r_{K}
	= -4 + 8\a + r_{K}
	\leq 10\a - 4,
\] 
as desired.
\end{proof}

\section{Proof of Theorem \ref{thm:main}}\label{sec:bound_beta_height}

\subsection{Bounding the Height of $\beta$}
Since $(\alpha_{2} - \alpha_{3})/(\alpha_{3} - \alpha_{1})$ is fixed, we can next bound the height of the unit $\beta_1/\beta_2$.

\begin{lemma}\label{lem:H_K-beta1beta2}
We have that
\[
	H_K\left(\frac{\beta_1}{\beta_2}\right)
	\leq 11\a -4.
\]
\end{lemma}
\begin{proof}
Let us denote the \textit{local height} by 
\[
	H_a(f)
	:= -\sum_{w\mid v_a}\min(0,w(f))
	,\quad a\in \C\cup \{\infty\}.
\] 
Then
\begin{equation}\label{height inequality}
	H_K(f)
	=\sum_{a\in\C\cup\{\infty\}}H_a(f)
	\geq H_\infty(f),
\end{equation}
and since $w(fg) = w(f)+w(g)$ for  each valuation, it follows that
\[H_a(fg)\leq H_a(f) + H_a(g)\]
 for any $f,g\in K$. Moreover, since $\beta_1/\beta_2$ is a unit in $\OO_K$, we have
\begin{equation}\label{eq:H_Kbeta-estimate}
	H_K \left( \frac{\beta_1}{\beta_2} \right)
	= H_\infty \left( \frac{\beta_1}{\beta_2} \right)
	\leq H_\infty \left(\frac{(\alpha_{2}-\alpha_{3})\beta_{1}}{(\alpha_{3}-\alpha_{1})\beta_{2}} \right) +
	H_\infty \left(\frac{\alpha_3-\alpha_1}{\alpha_2-\alpha_3} \right).
\end{equation}

In order to compute the last height in the above estimation, we recall that
\[
	\alpha_1 = 1 + \dots, \quad
	\alpha_2 = -\frac{1}{\lambda} + \dots, \quad
	\alpha_3 = -1 + \dots, \quad
	\alpha_4 = \lambda + \dots .
\]
Therefore 
\[w_1\left(\frac{\alpha_3-\alpha_1}{\alpha_2-\alpha_3}\right)=w_1(\alpha_3-\alpha_1) - w_1(\alpha_2-\alpha_3) = w_{1}(2+\dots)-w_{1}(1+\dots)=0.\]
Similarly, $\iota_{2}(\alpha_3-\alpha_1) = \alpha_4-\alpha_2 = \lambda + \dots$, i.e. $w_2(\alpha_3-\alpha_1)=-\a$, and 
$\iota_{2}(\alpha_2-\alpha_3)=\alpha_3-\alpha_4 = -\lambda + \dots$, i.e. $w_2(\alpha_2-\alpha_3)=-\a,$ which together yields
\[w_2\left(\frac{\alpha_3-\alpha_1}{\alpha_2-\alpha_3}\right)=- \a - (-\a) = 0.\]
Finally, we compute $w_3((\alpha_3-\alpha_1)/(\alpha_2-\alpha_3))= 0- (-\a) = \a,$ and $w_4((\alpha_3-\alpha_1)/(\alpha_2-\alpha_3))= -\a - 0 = -\a.$
It follows that
\[\left(\frac{\alpha_3-\alpha_1}{\alpha_2-\alpha_3}\right)_\infty=(0,0,\a,-\a),\]
and therefore that
\begin{equation}\label{height of alpha ratio}
	H_\infty\left(\frac{\alpha_3-\alpha_1}{\alpha_2-\alpha_3}\right)=\a. 
\end{equation}
By inequality \eqref{eq:H_Kbeta-estimate}, followed by  \eqref{height inequality} and \eqref{height of alpha ratio}, and finally Lemma \ref{lem:upperBound-H_K-gammas}, we conclude that
\[
H_K \left( \frac{\beta_1}{\beta_2} \right) \leq H_\infty \left(\frac{(\alpha_{2}-\alpha_{3})\beta_{1}}{(\alpha_{3}-\alpha_{1})\beta_{2}} \right) +
	H_\infty \left(\frac{\alpha_3-\alpha_1}{\alpha_2-\alpha_3} \right) 
	\leq H_K\left(\frac{(\alpha_{2}-\alpha_{3})\beta_{1}}{(\alpha_{3}-\alpha_{1})\beta_{2}}\right) +\a 
	\leq 11\a -4,
\]
as desired.
\end{proof}

Finally, we obtain a bound for the height of $\beta$.

\begin{lemma}\label{lem:H_K-bound}
We have that
\[
	H_K(\beta)\leq 11\a - 4.
\]
\end{lemma}
\begin{proof}
In the previous Lemma we obtained an upper bound for the height $H_K(\beta_1/\beta_2)$. Now we express it in a different way using the fact that $w_i(\beta_2)=w_i(\sigma(\beta_1))=w_{i+1}(\beta_1)$ (where, as always, $i+1$ is considered mod 4):
\begin{align*}
	H_K\left(\frac{\beta_1}{\beta_2}\right)
	&= - \sum_{i=1}^4 \min(0, w_i(\beta_1/\beta_2))
	= - \sum_{i=1}^4 \min(0, w_i(\beta_1)-w_i(\beta_2))\\
	&= \sum_{i=1}^4 \max(0,w_i(\beta_2)-w_i(\beta_1))
	= \sum_{i=1}^4 \max(0,w_{i+1}(\beta_1)-w_i(\beta_1)).
\end{align*}
In order to compute this sum, let us define $b_1,b_2,b_3,b_4$ such that
\[
	\{b_1,b_2,b_3,b_4\}=\{w_1(\beta),w_2(\beta),w_3(\beta),w_4(\beta)\}
	\quad \text{and} \quad
	b_1\leq b_2 \leq b_3 \leq b_4.
\]
Let $\psi$ be the permutation that maps the coefficients $\{1,2,3,4\}$ of the $w(\beta)$'s to the coefficients of the $b$'s, i.e.\ $\psi \colon \{1,2,3,4\} \to \{1,2,3,4\}$ such that
\[
	w_i(\beta)=b_{\psi(i)},
	\quad i=1,2,3,4.
\]
Next, we want to have a map $\varphi$ for the coefficients of the $b$'s such that if $b_i=w_j(\beta)$, then $b_{\varphi(i)}=w_{j+1}(\beta)$. Therefore, we define $\varphi \colon \{1,2,3,4\} \to \{1,2,3,4\}$,
\[
	\varphi(i) = \psi(\psi^{-1}(i)+1).
\]
Since $\psi$ is a bijection and $j \mapsto j+1 \pmod 4$ is a 4-cycle, it is clear that $\varphi$ is also a 4-cycle. Note that there exist 6 different 4-cycles.

Now we can use this notation to rewrite $H_K(\beta_1/\beta_2)$ and compute it:
\begin{align*}
	H_K\left(\frac{\beta_1}{\beta_2}\right)
	&= \sum_{j=1}^4 \max(0,b_{\varphi(j)}-b_j)\\
	&=\begin{cases}
	b_4 - b_1 & \text{if } \varphi \in \{(1234),(1243),(1342),(1432)\},\\
	b_4 - b_1 + b_3-b_2 &\text{if } \varphi \in \{(1324),(1423)\}.
	\end{cases}
\end{align*}
In any case,
\[
	H_K\left(\frac{\beta_1}{\beta_2}\right)
	\geq b_4-b_1,
\]
which together with Lemma \ref{lem:H_K-beta1beta2} yields
\[
	b_4-b_1
	\leq 11\a - 4.
\]

Note that $H_K(\beta)=H_K(\beta^{-1})$ by the product formula, and thus we may assume that either $b_1<0$ and $0\leq b_2\leq b_3 \leq b_4$ or $b_1\leq b_2 < 0$ and $0\leq b_3 \leq b_4$ (otherwise just consider $\beta^{-1}$ instead of $\beta$).

\textit{Case 1:} $b_1<0$ and $0\leq b_2\leq b_3 \leq b_4$. Then we obtain
\[
	H_K(\beta)
	= -b_1
	\leq - b_1 + b_4
	\leq 11\a - 4.
\]

\textit{Case 2:} $b_1\leq b_2 < 0$ and $0\leq b_3 \leq b_4$. Note that $2(-b_2)\leq -b_1 -b_2 =b_3 + b_4 \leq 2 b_4$, so $-b_2\leq b_4$. Thus we obtain
\[
	H_K(\beta)
	= (-b_1) + (-b_2)
	\leq -b_1 + b_4
	\leq 11\a - 4.
\]

In both cases we have proven the required upper bound.
\end{proof}
\subsection{Completion of Proof}
Finally, we proceed to the proof of Theorem \ref{thm:main}.

\begin{proof}[Proof of Theorem \ref{thm:main}]

Since $\beta \in \C[T][\alpha]^\times$ is a unit, by Proposition \ref{prop:fundSystem} it can be written as
\[
	\beta = \eta (\alpha-1)^r\alpha^s (\alpha+1)^t,
\] 
with $\eta \in \C^\times$ and $r,s,t \in \Z$.
Thus, together with Lemma \ref{lem:H_K-bound} we obtain
\begin{align*}
	11\a - 4
	\geq H_K(\beta)
	&= - \sum_{i=1}^4 \min (0, w_i(\eta (\alpha-1)^r\alpha^s (\alpha+1)^t))\\
	&= \sum_{i=1}^4 \max(0, - (w_i(\eta) + r w_i(\alpha-1) + s w_i(\alpha) + tw_i(\alpha+1))).
\end{align*}
Note that $w_i(\eta)=0$ for $i=1,2,3,4$, and recall that $(\alpha-1)_\infty=(\a,0,0,-\a)$, $(\alpha)_\infty=(0,\a,0,-\a)$ and $(\alpha+1)_\infty=(0,0,\a,-\a)$.
It follows that
\begin{align*}
	11\a-4
	\geq H_K(\beta)
	&=  \max(0,-r \a) + \max (0,-s \a) + \max(0,-t\a) + \max (0, (r+s+t)\a).
\end{align*}
This implies
\begin{equation}\label{eq:upper_bound_on_units}
	\max(0,-r) + \max(0,-s) + \max(0,-t) + \max(0,r+s+t) \leq 11-\frac{4}{\a}< 11.
\end{equation}
In particular, for each  $(r,s,t)\in \Z^3$ which satisfies the above inequality, we have that $|r|,|s|,|t|\leq 10$.  This is a (sufficiently small) finite set of values, and it remains to check which of the corresponding units $\beta = \eta (\alpha-1)^r\alpha^s (\alpha+1)^t \in \C[T][\alpha]^{\times}$
yield a solution $(x,y) \in S_{\lambda,\xi}$.  In particular, while a general unit is of the form $\beta = x_3 \alpha^3 + x_2 \alpha^2 + x_1 \alpha + x_0,$ where $x_0,x_{1},x_{2},x_{3} \in \C[T]$, we are interested in those units for which $x_{3} = x_{2}=0$, i.e. units of the form $\beta = x - \alpha y$, where $x, y \in \C[T].$
We implement these computations using Sage \cite{sagemath}, a code which is provided in the \nameref{sec:appendix} below.  In doing so, we find that the only relevant values $(r,s,t) \in \Z^{3}$ lie in the trivial set $\{(0,0,0),(1,0,0),(0,1,0),(0,0,1)\}$. Therefore, $\beta =x-\alpha y$ must lie in the set
\begin{multline*}
	\{\eta, \eta (\alpha-1), \eta \alpha, \eta(\alpha+1)
		\colon \eta \in \C^\times\} \\
	= 
\{ \eta - \alpha \cdot 0, -\eta - \alpha(- \eta), 0 - \alpha (-\eta) , \eta - \alpha (-\eta) \colon \eta \in \C^\times  
	\}
\end{multline*}
which implies that 
\begin{align*}
	(x,y) \in &\{(\eta,0),(-\eta,-\eta),(0,-\eta),(\eta,-\eta) \colon \eta \in \C^\times\}\\
&= \{(\eta,0),(\eta,\eta),(0,\eta),(\eta,-\eta) \colon \eta \in \C^\times\}.
\end{align*}
We have shown that any possible solution $(x,y)\in S_{\lambda,\xi}$ must lie in the above set. Plugging into $F_\lambda(X,Y)=\xi$, we find that the full solution set is indeed 
\[S_{\lambda,\xi}=\{(\eta,0),(0,\eta): \eta^{4} = \xi \} \cup \{(\eta,\eta), (\eta,-\eta): -4\eta^{4} = \xi \},
\] 
as desired.

\end{proof}

\section*{Appendix}\label{sec:appendix}
The following Sage code outputs the units $\beta =\eta (\alpha-1)^r\alpha^s (\alpha+1)^t \in \C[T][\alpha]^{\times}$ such that $(r,s,t) \in \Z^{3}$ satisfy \eqref{eq:upper_bound_on_units} and such that $\beta$ is of the form $\beta = x-\alpha y$, for $x,y \in \C[T]$.  The code may be run in less than a minute on a standard computer.  Note that although the computations technically take place in an extension of $\Q(L)$ (where $L$ is a stand-in for $\lambda$) they are exactly the same as when performed in $\C(T)(\alpha)$.
\begin{lstlisting}[language=Python]
Q.<L> = FunctionField(QQ)
R.<x> = Q[]
K.<alpha> = Q.extension(x^4 -L*x^3-6*x^2+L*x+1)

for r in range(-10, 10 + 1):
    for s in range(-10 + max(0, -r), 10 - max(0, r) + 1):
        for t in range(-10 + max(0, -r) + max(0, -s), 
                       10 - max(0, r) - max(0, s) + 1):
            beta = (alpha-1)^r * alpha^s * (alpha+1)^t
            betacoeff = beta.matrix()[0]
            if betacoeff[3] == 0 and betacoeff[2] == 0:
                print(r,s,t)
\end{lstlisting}

\bibliographystyle{habbrv}
\bibliography{ThueReferences}

\end{document}